\documentclass[12pt]{amsart}
\usepackage{amsrefs}
\usepackage[all,cmtip]{xy}
\usepackage{amssymb}
\usepackage{hyperref}
\usepackage{fancyhdr}
\usepackage{xcolor}
\usepackage{graphics}
\usepackage{bbm}
\usepackage[export]{adjustbox}
\usepackage{latexsym}

\usepackage{amsthm}
\usepackage{amscd}
\usepackage{amsmath}
\usepackage{amsfonts}
\usepackage{graphics}

\newtheorem{theorem}{Theorem}
\newtheorem*{theorem*}{Theorem}
\newtheorem*{proposition*}{Proposition}
\newtheorem{lemma}[theorem]{Lemma}
\newtheorem{proposition}[theorem]{Proposition}
\newtheorem{claim}[theorem]{Claim}
\newtheorem*{claim*}{Claim}
\newtheorem{fact}[theorem]{Fact}
\newtheorem{corollary}[theorem]{Corollary}

\newtheorem{remark}[theorem]{Remark}

\newtheorem*{namedtheorem}{\theoremname}
\newcommand{\theoremname}{testing}
\newenvironment{named}[1]{\renewcommand{\theoremname}{#1}\begin{namedtheorem}}{\end{namedtheorem}}
\theoremstyle{definition}
\newtheorem{definition}[theorem]{Definition}
\newtheorem*{definition*}{Definition}
\newtheorem*{lemma*}{Lemma}
\usepackage{graphicx}
\usepackage{pstricks, enumerate, pst-node, pst-text, pst-plot}

\numberwithin{equation}{section}
\numberwithin{theorem}{section}

\newcommand{\CG}{\mathcal{G}}

\newcommand{\R}{\mathbb{R}}

\newcommand{\subg}{\mathrm{Sub}_G}
\newcommand{\cosg}{\mathrm{Cos}_G}
\newcommand{\cM}{\mathcal{M}}
\newcommand{\supp}[1]{\mathrm{supp}\left({#1}\right)}
\newcommand{\stab}{\mathrm{stab}}
\newcommand{\comment}[1]{}

\def\dd{\mathrm{d}}

\begin{document}

\title[]{Unimodularity of Invariant Random Subgroups}

\author{Ian Biringer}
\address[I.\ Biringer]{Department of Mathematics, Boston College,
  Chestnut Hill MA, 02467, USA.}
\author{Omer Tamuz}
\address[O.\ Tamuz]{Microsoft Research New England, One Memorial Drive, Cambridge
  MA, 02142, USA.}

\date{\today}

\begin{abstract}
  An invariant random subgroup $H \leq G$ is a random closed subgroup
  whose law is invariant to conjugation by all elements of $G$. When
  $G$ is locally compact and second countable, we show that for every
  invariant random subgroup $H \leq G$ there almost surely exists an
  invariant measure on $G/H$. Equivalently, the modular function of
  $H$ is almost surely equal to the modular function of $G$,
  restricted to $H$.

  We use this result to construct invariant measures on orbit
  equivalence relations of measure preserving actions. Additionally,
  we prove a mass transport principle for discrete or compact invariant
  random subgroups.
\end{abstract}

\maketitle

\section{Introduction}

Let $G$ be a locally compact, second countable group. We denote by
$\subg$ the space of closed subgroups of $G$, equipped with the
Chabauty topology, and consider the action $G \curvearrowright \subg$
by conjugation.

An {\em invariant random subgroup} (IRS) is a $\subg$-valued random
variable whose law is invariant to conjugation. Any Borel
probability measure preserving action $G \curvearrowright X$ gives an
IRS - the stabilizer of a random point in $X $; this
follows\footnote{This theorem states that given a measurable
  action of a locally compact second countable group on a standard
  probability space $X$, one can endow $X$ with a precompact metrizable
  topology under which the action is continuous. This implies that
  stabilizers are closed subgroups, and that the map that assigns to
  each point its stabilizer is upper semicontinuous and hence
  measurable.}  from the Compact Model Theorem of
Varadarajan~\cite{varadarajan1963groups}*{Theorem 3.2 and its
  corollary}.  In a slight (but standard) abuse of notation we also
refer to a conjugation invariant Borel probability measure on $\subg$
(i.e., the law of an IRS) as an IRS.

Although measure preserving $G$-actions and their stabilizers have
been studied for some time (e.g.,~\cites{stuck1994stabilizers,
  bergeron2004asymptotique}), IRSs were first introduced by Ab{\'e}rt,
Glasner and Vir{\'a}g~\cite{abert2012kesten}, and simultaneously by
Vershik~\cite{vershik2012totally} under a different name.  Since then,
IRSs have appeared in a number of papers, either as direct subjects of
study~\cites{vershik2012totally,bowen2012invariant,
  bowen2012invariantb, creutz2013stabilizers-of-ergodic,
  hartman2013stabilizer}, as probabilistic limits of manifolds with
increasing volume \cite{abert2012growth}, or as tools to understand
stationary group actions~\cites{bowen2010random,
  hartman2012furstenberg}.

The notion of an IRS is a natural weakening of that of a normal
subgroup. As such, it is interesting to understand which properties of
normal subgroups hold for IRSs. This is the spirit
of~\cite{abert2012kesten}, and of our main theorem.

More generally, one can consider the set of $G$-invariant,
$\sigma$-finite Borel measures on $\subg$, which we denote by
$\cM_G(\subg)$. We will call an element of $\cM_G(\subg)$ an {\em
  invariant subgroup measure}.  If $(*)$ is a Borel property of subgroups of
$G$ (e.g.\ `discrete' or `unimodular') then we say that $\lambda \in
\cM_G(\subg)$ is $(*)$ whenever $\lambda$-a.e.\ $H \in \subg$ is
$(*)$.

Our main result is the following theorem.
\begin{theorem}
  \label{thm:unimodular}
  Let $\lambda \in \cM_G(\subg)$ be an invariant subgroup
  measure. Then for all $H$ in the topological support of $\lambda$
  there exists a nontrivial $G$-invariant measure on $G/H$.
\end{theorem}

There exists a $G$-invariant measure on $G/H$ if and only if the
modular function $\mu_G$ of $G$ restricts to the modular function
$\mu_H $ of $H $ (cf.~\cite[Theorem 1, pg 139]{nachbin1965}).  Hence a corollary of
Theorem~\ref{thm:unimodular} is that an invariant subgroup measure of
a unimodular group $G$ must be supported on unimodular subgroups.

When $G$ is not unimodular, the kernel of the modular function is a
unimodular, closed, normal proper subgroup of $G$.  It can be thought
of as the `unimodular radical' of $G$: a normal unimodular subgroup
that includes all other normal unimodular subgroups.  Theorem
\ref{thm:unimodular} implies that any \emph{unimodular} subgroup
$H\leq G$ that lies in the support of an invariant subgroup measure
must be contained in $\ker \mu_G$.  That is,

\begin{corollary}
  \label{cor:modular}
  If an invariant subgroup measure $\lambda \in \cM_G(\subg)$ is unimodular, then $H \leq \ker \mu_G$ for all $H$ in the support of $\lambda$.
\end{corollary}

\noindent Note that, in particular, this corollary applies whenever
$\lambda$ is supported on discrete subgroups.  It would be interesting
to understand when the same result can be proved for other
`radicals'. For example, in~\cite{bader2014amenable} it is shown that
every amenable IRS is included in the amenable radical.

\vspace{2mm}

Our proof of Theorem~\ref{thm:unimodular} yields a slightly stronger
theorem (Theorem~\ref{stronger}), in which the $G$-invariant measures
on the coset spaces $G/H$ can be chosen to vary measurably with $H $.
As an application, we construct invariant measures on orbit
equivalence relations.

Suppose that $G \curvearrowright (X,\zeta)$ is a measure preserving
action on a $\sigma$-finite Borel measure space.  Let $E \subseteq
X \times X$ denote the orbit equivalence relation
$$E = \{(x,g x)\,:\,x \in X, g \in G\},$$ 
which we consider with the $G$-action $g(x,y) = (x,g y)$.  Let
$p_l \colon E \to X$ be the projection on the left coordinate 
$p_l(x,y) = x$. 

In many applications the interesting case is when $\zeta$ is a
probability measure. More generally, we prove:
\begin{corollary}
  \label{cor:actions}

  Let $G \curvearrowright (X,\zeta)$ be a measure
  preserving action on a $\sigma $-finite Borel measure space, with
  orbit equivalence relation $E$.

  Then there exists a Borel family of $G$-invariant, $\sigma$-finite
  measures $\nu_x$ on $E$, with $\nu_x$ supported on the fiber
  $p_l^{-1}(x)$, and such that
  \begin{align*}
    \nu = \int_X\nu_x\,\dd\zeta(x)
  \end{align*}
  is a $G$-invariant, $\sigma$-finite measure on
  $E$.
\end{corollary}

When the action is free, each fiber $p_l^{-1}(x)$ is identified with
$G $ and one can take $\nu_x$ to be the Haar measure. When $G$ is
countable, the $\nu_x$ are counting measures and Corollary
\ref{cor:actions} is due to Feldman and Moore \cite [Theorem
2]{feldman1975ergodic}. Such a measure $\nu$ can be used, for example,
to define the so called groupoid representation $G \curvearrowright
L^2(\nu)$ (see~\cite{renault1980groupoid},
and~\cite{vershik2011nonfree} for a recent application).

The final goal of this paper is to formulate a version of the `mass
transport principle' (MTP) for invariant subgroup measures.  The MTP
has proved to be a useful tool in the theory of random graphs
\cites{Benjaminirecurrence,Aldousprocesses} and in percolation theory
\cites{Benjaminigroup,Haggstrominfinite}.  Versions of the MTP have
also appeared in the theory of foliations for some time (see, e.g., the
definition of an invariant transverse measure in~\cite[page 82]{Mooreglobal}).

A {\em mass transport} on a graph $\Gamma=(V,E)$ is a positive function $f
\colon V \times V \to \R$, where $f(x,y)$ determines how much ``mass''
is transported from $x$ to $y$.  We assume that $f$ is invariant to
the diagonal action of the automorphism group $\mathrm{Aut}(\Gamma)$, so
that $f(x,y)$ depends only on how $x$ and $y$ interact with the
geometry of the graph.  When $\mathrm{Aut}(\Gamma)$ is transitive and
unimodular, a {\em mass transport principle} (MTP) applies: the total
amount of mass transported from each vertex is equal to the amount
transported to it.  Namely, for all $v \in V$,
\begin{align}
  \sum_{w \in V}f(v,w) = \sum_{w \in V}f(w,v). \label{transitiveMTP}
\end{align}

The notion of a mass transport principle can be generalized to random
rooted graphs.  Let $\CG_\bullet$ be the space of isomorphism classes
of rooted graphs, equipped with the topology of convergence on finite
neighborhoods. Similarly, let $\CG_{\bullet\bullet}$ be the space of
isomorphism classes of doubly rooted graphs
(see~\cite{Benjaminirecurrence}). A random rooted graph `satisfies the
mass transport principle' when its law, a measure $\lambda$ on
$\CG_\bullet$, satisfies
\begin{align}
  \label{eq:graph-mtp}
  \int_{\CG_\bullet}\ \sum_{w\in \Gamma}f(\Gamma,v,w)\, \dd\lambda(\Gamma,v) =
  \int_{\CG_\bullet}\ \sum_{w\in \Gamma}f(\Gamma,w,v)\, \dd\lambda(\Gamma,v),
\end{align}
for all positive Borel functions $f: \CG_{\bullet\bullet}
\longrightarrow \R$.  Random rooted graphs that satisfy the MTP are
also called \emph {unimodular random graphs}.  Let a random rooted
graph $\Gamma$ be almost surely a single transitive rooted graph. Then
$\Gamma$ is unimodular if and only if the automorphism group
$\mathrm{Aut}(\Gamma)$ is a unimodular topological group, in which
case~\eqref{eq:graph-mtp} is exactly the MTP given in~\eqref
{transitiveMTP}.

One should view an MTP as a replacement for group-invariance of a
measure, which is especially useful when no group action is present.
Our interest, however, lies with measures on $\subg$, where $G$ acts
by conjugation.  We show that for measures supported on discrete or
compact subgroups of $G $, conjugation invariance is equivalent to a
suitable MTP. This generalizes the results of Ab{\'e}rt, Glasner and
Vir{\'a}g \cite{abert2012kesten} for discrete $G$. We elaborate on
this connection in Section~\ref{sec:mtp}. Discrete IRSs are
particularly interesting in the case of Lie groups: the non-atomic
IRSs of a connected simple Lie group are supported on discrete
groups~\cite{abert2012growth}*{Theorem 2.6}.

Here, we will state our theorem in the case that $G $ is unimodular,
saving the more general statement for Section~\ref{sec:mtp}.  So,
suppose that $G $ is a unimodular, locally compact, second countable
group, and fix a Haar measure $\ell$ on $G $.  If $H$ is a compact
subgroup of $G $, let $\nu_H$ be the push forward of $\ell$ to
$H\backslash G$, while if $H $ is a discrete subgroup of $G$, let
$\nu_H$ be the measure on $H\backslash G$ obtained by \emph {locally
  pushing forward} $\ell$ under the covering map $G \longrightarrow
H\backslash G$. This is the unique measure with
\begin{align}
  \label{pushforward1}\ell = \int_{G/H}\eta_{H g} \, \dd\nu_H(H g),
\end{align}
where $\eta_{H g}$ is the counting measure on $H g$.  See Section~\ref{sec:mtp} for details.

Denote by $\cosg$ the space of cosets of closed subgroups of $G$:
\begin{align*}
  \cosg = \{Hg \,:\,  H \in \subg, \, g\in G\}.
\end{align*}
We discuss this space and its topology in Section~\ref{sec:disintegration}.
\begin{theorem}[Mass Transport Principle]
  \label{thm:mtp-intro}
  Let $G$ be unimodular, and let $\lambda$ be a $\sigma$-finite Borel
  measure on $\subg$ such that $\lambda$-a.e.\ $H\in\subg $ is
  discrete or compact.  

Then $\lambda $ is conjugation invariant
  (i.e., an invariant subgroup measure) if and only if for every nonnegative Borel
  function $f \colon \cosg \rightarrow \R$,

  \begin{align*}
    &  \int_{\subg} \int_{H\backslash G} f(Hg)\, \dd\nu_H(Hg) \,
    \dd\lambda(H) \nonumber \\
    &=\int_{\subg} \int_{H\backslash G}f(g^{-1}H) \,
    \dd\nu_H(Hg)\, \dd\lambda(H).
  \end{align*}
\end{theorem}


Theorem~\ref{thm:mtp-intro} will be used by Ab\'ert and Biringer
in~\cite{abert-biringer} to show that certain invariant random
subgroups of continuous groups correspond to `unimodular random
manifolds', i.e.\ measures on the space of rooted Riemannian manifolds
satisfying a mass transport principle.

We should note that this is not the first version of a mass transport
principle that applies in the continuous setting.  In
\cite{benjamini2001percolation}, Benjamini and Schramm give a version
of the MTP for the hyperbolic plane.

To interpret Theorem \ref{thm:mtp-intro}, one should view $\cosg$ as
foliated by the right coset spaces $H \backslash G$, where $H $ ranges through
$\subg$.  The MTP says that the measure $\nu $ obtained by integrating
the measures $\nu_H$ against $\lambda $ is invariant under the
involution $Hg \mapsto g^{-1}H $ of $\cosg$.

An alternative, appealing interpretation of the MTP is the following.
Call a closed normal subgroup $N \lhd G$ {\em co-unimodular} if $G/N$
is unimodular; that is, if there exists a bi-invariant measure on
$G/N$.  Analogously, we call an IRS $\lambda$ co-unimodular if there
exists a Borel measure $\nu$ on $\cosg$ that projects to $\lambda$ and
is invariant to both the left and right $G$-action. In
Section~\ref{sec:bi-invariant} we show that when $G$ is unimodular,
then there is an MTP for $\lambda$ if and only if $\lambda$ is
co-unimodular.

A particularly aesthetic version of the MTP arises if we also define
for each discrete $H\leq G$, a measure $\nu^H$ on $G/H$ by
locally pushing forward $\ell$ with respect to the covering map $G
\longrightarrow G/H$.  
That is,
\begin{align}
  \label{pushforward2}\ell = \int_{H\backslash G}\eta_{g H} \,
  \dd\nu^H(g H),
\end{align}
where $\eta_{g H} $ is the counting measure on $Hg $.  It follows
from~\eqref{pushforward1} and~\eqref {pushforward2} that the
involution $Hg \mapsto g^{-1}H $ sends $\nu _ H $ to $\nu^H$. So, the
MTP can be rephrased as
\begin{align*}
   & \int_{\subg} \int_{H \backslash  G}f(Hg) \, \dd\nu_H(Hg)\, \dd\lambda(H) \\
&= \ \ \ \int_{\subg} \int_{G/H}  f(g H)\, \dd\nu^H(g H) \, \dd\lambda(H).
\end{align*}
In other words, the measure on $\cosg$ obtained by integrating the
natural measures on right coset spaces $H \backslash G $ against $\lambda$ is the same as the measure on $\cosg$ obtained by integrating the natural
measures on left coset spaces $G/H$ against $\lambda $.

\subsection{Acknowledgments}

The authors are indebted to Lewis Bowen, who first posed the question
that led to Theorem~\ref{thm:unimodular}, suggested that we include
	the statement of Theorem \ref{cor:actions}, and inspired the discussion in Remark \ref{longremark}.  The first author is
partially supported by NSF grant DMS-1308678 and would like to thank
Miklos Ab\'ert for numerous conversations, in particular those
relating to the mass transport principle, as without his input the
statement given here might not have been considered or solved.  The
second author would like to thank Yair Hartman for enlightening
discussions. Both authors would also like to thank the referee for greatly improving the readability and accuracy of the paper.

\section{Cosets, subgroups and a disintegration result}
\label{sec:disintegration}

We start by defining the coset space of $G$, $\cosg$, which is the set
\begin{align*}
  \cosg = \{g H \,:\, g\in G,\, H \in \subg\},
\end{align*}
equipped with the Fell topology of closed subsets of $G$. This topology is locally compact, second countable and Hausdorff, hence Polish. Note that
since $g H = g H g^{-1}g = H^g g$, an equivalent definition is
\begin{align*}
  \cosg = \{Hg \,:\, g\in G,\, H \in \subg\}. 
\end{align*}
The group $G $ acts on $\cosg$ from the left.  When viewing $\cosg$ as
the set of left cosets, the action of $k\in G$ is given by $k(g H) =
k g H$, and when considering right cosets we have $k(Hg) =
(k H k^{-1})(kg) = H^k k g$.  

Adopting the perspective of right cosets, consider the maps
\begin{equation*}
  \begin{array}{r c l  c l}
    \subg \times G & \overset{\sigma_r}{\longrightarrow} &\cosg & \overset {\pi_r}{\longrightarrow} & \subg,\\
     (H, g)&\longmapsto     &H g &\longmapsto     & H. 
\end{array}
\end{equation*}
These maps are all $G$-equivariant, where the actions are
$(H,g) \overset {k}{\mapsto} (k H k^{-1},kg)$, $Hg \overset {k}{\mapsto}
k H k^{-1}(kg)$, and $H \overset {k}{\mapsto} k H k^{-1}$.  Note that
$\pi_r \circ \sigma_r$ is the projection $(H,g) \mapsto H$.  We will
also need the maps
\begin{equation*}
  \begin{array}{r c l c  l}
    \subg \times G & \overset{\sigma_l}{\longrightarrow} &\cosg & \overset {\pi_l}{\longrightarrow} & \subg,\\
     (H, g)&\longmapsto     &g H &\longmapsto     & H. 
\end{array}
\end{equation*}
When viewing $\cosg$ as right cosets, $\pi_l(Hg) = \pi_l(g\cdot  g^{-1}Hg) = g^{-1}Hg.$

\vspace {2mm}

Suppose that $G $ is a locally compact, second countable group. The main result of this section is the following disintegration theorem.

\begin {proposition}\label{prop:constructing-nu}
  Fix a left Haar measure $\ell$ on $G$, and a Borel map
  $$\cosg \longrightarrow \cM(G), \ \ \ Hg \longmapsto \eta_{Hg} $$
  such that each $\eta_{Hg}$ is nonzero, left $H$-invariant and
  supported on $Hg $.
  
  Then for each $H\in\subg$ there is a unique Borel measure $\nu_H$ on
  $H \backslash G\subset \cosg$ such that $$ \ell=\int_{H \backslash
    G} \eta_{Hg} \, \dd \nu_H(Hg).$$ Moreover, the map $\subg
  \longrightarrow \cM(\cosg), \ H \longmapsto \nu_H$ is Borel.
\end {proposition}

Succinctly, the proposition states that left $H$-invariant measures on
cosets $Hg$ can be realized as fiber measures in a disintegration of
any left Haar measure on $G$, and that if the fiber measures are Borel
parametrized, so are the resulting factor measures.

\begin {proof}
  Since $\ell$ is $\sigma$-finite, the push forward measure on $H
  \backslash G$ is equivalent to a $\sigma$-finite measure $\hat
  \nu_H$; for instance, one can take $\hat \nu_H$ to be the push
  forward of any probability measure on $G$ that is equivalent to
  $\ell$. Note that the push forward of $\ell$ is only itself
  $\sigma$-finite if $H$ is compact.

  Applying Rohlin's Disintegration Theorem
  (see~\cite{Simmonsconditional}*{Theorem 6.3}), there is a
  disintegration
  \begin{align} \label{firstdis} \ell =
    \int_{H \backslash G}\hat\eta_{Hg}\,\dd\hat\nu_H(Hg),
  \end{align}
  where $Hg\longmapsto \hat\eta_{Hg}$ is a measurable parametrization of
  $\sigma$-finite measures on $G $, and $\hat\eta_{Hg}$ is supported on the coset $H g$.
  
  We claim that $\hat\nu_H$-almost every $\hat\eta_{Hg}$ is left
  $H$-invariant: left multiplication by $h \in H$ on $G$ leaves
  invariant each right coset $Hg$, and so
  $$\ell = h_*(\ell)  = \int_{H \backslash G} h_*(\hat\eta_{Hg}) \,
  \dd\hat\nu_H(Hg).$$  By the uniqueness of
  disintegrations with a given factor measure, we have
  $h_*(\hat\eta_{Hg})=\hat\eta_{Hg}$ for $\hat\nu_H$-a.e.\ coset $Hg$.  It follows
  from the separability of $H$ that $\hat\eta_{Hg}$ is $H$-invariant for
  $\hat\nu_H$-a.e.\ $Hg$.  

  Equation \eqref{firstdis} remains unchanged if we replace a
  $\hat\nu_H$-negligible number of the measures $\hat\eta_{Hg}$ with
  $\eta_{Hg}$, so we may as well assume from now on that every
  $\hat\eta_{Hg}$ is nonzero and left $H$-invariant. By the uniqueness
  of the Haar measure, each $\hat\eta_{Hg}$ is a positive scalar
  multiple of $\eta_{Hg}$. Let
  \begin{align}\label{eq:f}
    f : \cosg \to \R^+, \ \ f(Hg) = \frac{\dd\hat\eta_{Hg}}{\dd\eta_{Hg}},
  \end{align}
  be the function whose value gives this multiple. Define the measure
  $\nu_H$ on $H \backslash G$ by
  \begin{align*}
    d\nu_H(Hg) = f(Hg) \, d\hat\nu_H(Hg).
  \end{align*}
  Then we have
  \begin{align} \label{productdisintegration}
    \ell &= \int_{H \backslash G}\hat\eta_{Hg}\, \dd\hat\nu_H(Hg) \nonumber\\
    &= \int_{H \backslash G} \eta_{Hg} \cdot f(Hg) \, \dd\hat\nu_H(Hg) \nonumber\\
    &= \int_{H \backslash G}\eta_{Hg}\, \dd\nu_H(Hg), 
  \end{align}
  as required in the statement of the proposition. As the $\eta_{Hg}$
  are fixed, $\nu_H$ is the unique measure on $H \backslash G$
  satisfying \eqref{productdisintegration}.
  
  It remains to show that the $\nu_H$ are Borel parametrized, when
  regarded as measures on the space of all cosets $\cosg$. This is a
  consequence of the fact that the $\eta_{Hg}$ are Borel
  parametrized, but we will never-the-less give a careful proof.

  Fix a positive, continuous function $F : G \longrightarrow\R $ such
  that $$\int F \, \dd\eta_{Hg} < \infty \ \ \forall Hg \in \cosg.$$
  One way to produce such a function is as follows.  Pick a compact
  neighborhood $B$ of the identity, and a locally finite cover of $G$
  by open sets $B_n$ with $B_n^{-1}B_n \subset B$. If $h g \in Hg \cap
  B_n^{-1}$, then
  \begin {align*}
    \eta_{Hg}(B_n) &=\left( (h g)^{-1}_*\eta_{Hg}\right )(h g B_n)
    \\ &=\eta_{g^{-1}Hg}( h g B_n) \\ &< \eta_{g^{-1}Hg}(B).
  \end{align*}
  So, for each $Hg$, there is an upper bound on $\eta_{Hg}(B_n)$ that
  is independent of $n$. The function $F=\sum_n \frac{1}{2^n} \cdot
  \rho_n $, where $\rho_n$ is a partition of unity subordinate to the
  cover $B_n$, then has the desired properties.

  Using $F$, we define a positive, Borel function
  $$F' : \subg \times G \longrightarrow\R, \ \ F'(H,g)=\frac {F(g)}{\int F  \, \dd\eta_{Hg} }.$$
  By definition, this $F'$ has the property that $\int F'(H,\cdot) \,
  \dd\eta_{Hg} =1$ for all $Hg\in \cosg$. Given a continuous function
  $\varphi :\cosg \longrightarrow \R$, we have
  \begin {align*}
    \int \varphi(Hg) \, \dd\nu_H(Hg) &= \int \varphi(Hg) \left( \int_{k\in Hg} F'(H,k) \, \dd \eta_{Hg}(k)\right ) \, \dd \nu_H(Hg) \\
    &=\int \varphi(Hg) F'(H,g)\, \dd \ell(g).
  \end{align*}
  The last expression is Borel in $H$, so we have shown that
  integrating a fixed continuous function $\varphi$ against the
  measures $\nu_H$ gives a Borel function in $H$, implying that the
  map $H \longmapsto \nu_H$ is Borel.
\end {proof}

Finally, let us discuss a convenient construction of such $\eta_{Hg}$.  In Claim \ref{haarmeasures}, we show that there is a continuous map $$m
: \subg \longrightarrow \cM(G)$$ such that each $m (H)$ is a left Haar
measure on $H\subset G$. Define
$$ \eta_{Hg} := g_*m(g^{-1}Hg).$$ Each $\eta_{Hg}$ is a $\sigma$-finite Borel measure supported on the coset
$Hg \subset G$, and if $h\in H$ we have 
$$h_*(\eta_{Hg})
=h_*g_* m(g^{-1}H g)
= g_*(g^{-1} h g)_*m(g^{-1}H g)
=\eta_{H g},$$
which shows that $\eta_{H g}$ is left $H $-invariant,
and is well defined in the sense that it depends only on the coset $H
g$ and not on the representative $g$.
The $\eta_{Hg}$ are also permuted by the left $G $-action: if $k\in G$ then 
\begin{align*}
k_*(\eta_{Hg})=(k g)_*m(g^{-1}H g)
= \eta_{(k H k^{-1})k g} 
= \eta_{k H g}.
\end{align*}

So, to summarize this construction:

\begin {fact}\label{etas}
There is a family of measures $\eta_{Hg}$ as required by Proposition \ref{prop:constructing-nu} such that the map $Hg \longmapsto\eta_{Hg}$ is left $G$-equivariant: 
$$k_*(\eta_{Hg})=\eta_{k H g}, \ \ \ \forall k\in G, \ Hg \in \cosg.$$
\end {fact}

Note that for some classes of subgroups $H$, there are `natural' choices for a left Haar measure $m(H)$.  For instance, if $H$ is discrete, one can take $m(H)$ to be the counting measure, in which case the measures $\eta_{Hg}$ will all be counting measures as well. When $H$ is compact, one can take $m(H)$ to be the Haar probability measure. 

As `discrete' and `compact' are both Borel properties of subgroups, the map $H \longmapsto m(H)$ above (and therefore the measures $\eta_{Hg}$) can be adjusted to agree with these natural choices on such subgroups. This will be important in Section \ref{sec:mtp}.

\section{The proof of Theorem \ref{thm:unimodular}}
\label{proof}
\vspace{2mm}

Our main result is the following theorem.

\begin {theorem}\label{stronger}
  Let $\lambda \in \cM_G(\subg)$ be an invariant subgroup
  measure. Then there exists a Borel map 
  $$\subg \longrightarrow \mathcal M(\cosg), \ \ H\longmapsto \nu^ H $$
  such that for $\lambda$-a.e.\ $H$ the measure $\nu^H$ is nonzero,
  $G$-invariant and supported on the subset $G/H \subset \cosg$. 
\end {theorem}

As we show in Lemma~\ref{closed}, the set of all $H \in \subg$ such
that $G/H$ as an invariant measure is closed.  Therefore, we will
obtain as a corollary of Theorem~\ref{stronger} our main theorem from
the introduction:

\begin{named}{Theorem~\ref{thm:unimodular}}
  Let $\lambda \in \cM_G(\subg)$ be an invariant subgroup
  measure. Then for all $H$ in the topological support of $\lambda$
  there exists a nontrivial $G$-invariant measure on $G/H$.
\end{named}

We devote the remainder of this section to proving
Theorem~\ref{stronger}, and so fix an invariant subgroup measure
$\lambda \in \cM_G(\subg)$. 

\begin {proposition}
  \label{prop:equivalent}
  There is a left $G$-invariant, $\sigma$-finite Borel measure $\nu$ on
  $\cosg$ such that 
$\pi_{l*}\nu \equiv \lambda \equiv \pi_{r*}\nu.$
\end{proposition}

Here, recall that $\pi_l$ and $\pi_r$ are the maps $\cosg
\longrightarrow\subg$ taking $g H \longmapsto H$ and $Hg \longmapsto
H$, respectively. 

\vspace{2mm}

Before proving Proposition~\ref{prop:equivalent} we show that it
implies Theorem~\ref{stronger}, and in fact is equivalent to it: Given
Theorem~\ref{stronger}, the measure $\nu$ is obtained by integrating
against $\lambda $:
\begin {equation}\nu=\int_{\subg} \nu^H \, \dd \lambda(H).\label{eq:disintegrate-nu}\end {equation}
This $\nu$ is $G$-invariant since each $\nu^H$ is $G$-invariant, and $\pi_{l*}\nu \equiv \lambda$ since $\nu^H$ is nonzero and supported on $\pi_l^{-1}(H)$. 

Conversely, suppose Proposition~\ref{prop:equivalent} is true. Since $\pi_{l*}\nu \equiv
\lambda$, Rohlin's Disintegration Theorem
(see~\cite{Simmonsconditional}*{Theorem 6.3}) implies that there is a Borel map $H \longmapsto\nu^H \in \cM(\cosg)$, with $\nu^H$ supported on $G/H=\pi_l^{-1}(H)$, such that Equation \eqref{eq:disintegrate-nu} holds.

The action of $G$ on $\cosg$ leaves invariant all fibers of the map
$\pi_l:\cosg \longrightarrow \subg$.  So since $\nu$ is $G$-invariant,
the fiber measures $\nu^H$ are $G$-invariant for $\lambda$-a.e.\ $H
\in \subg$.  In other words, for $\lambda$-a.e.\ $H$, $\nu^H$ is an
invariant measure on $G/H$, proving Theorem~\ref{stronger}.

\begin {figure}[t]
\includegraphics [width=4.5in,left]{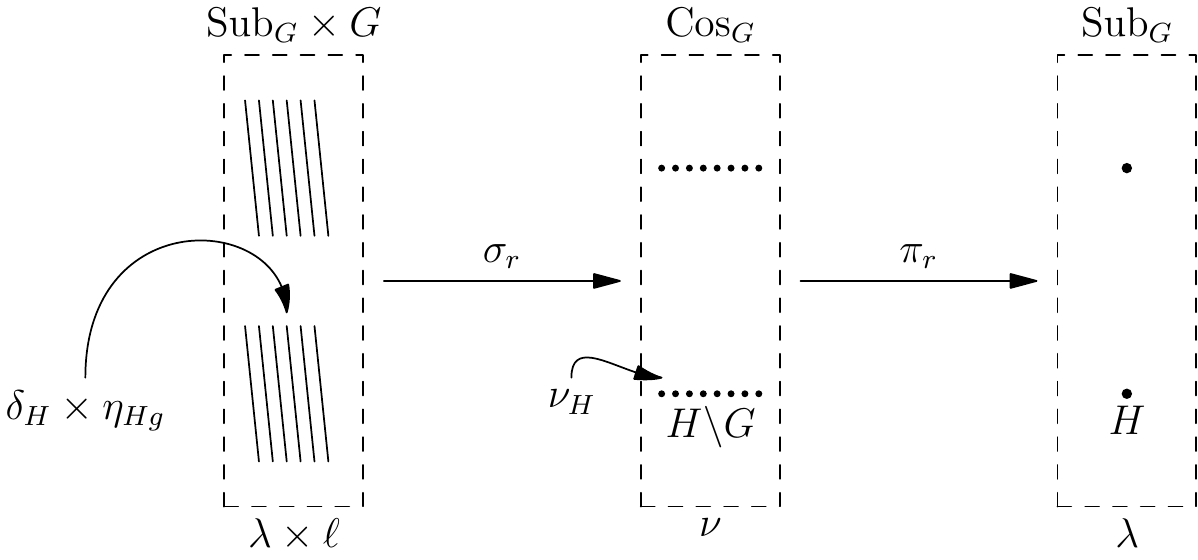}
\caption {The measures involved in the proof of Proposition~\ref{prop:equivalent}.}
\end {figure}

\begin{proof}[Proof of Proposition~\ref{prop:equivalent}]
  Fix a left Haar measure $\ell$ on $G$. Choosing a family $\eta_{Hg}$ of measures as described in Fact \ref{etas}, Proposition
  \ref{prop:constructing-nu} gives a Borel family of
  measures $\nu_H$ on $H \backslash G \subset\cosg$ with
  $$ \ell=\int_{H \backslash G} \eta_{Hg} \, \dd \nu_H(Hg)$$ for each $H\in \subg$, and we set $$\nu = \int \nu_H \, \dd \lambda(H).$$ 
  Note that $\pi_{r*}\nu \equiv \lambda $, since $\nu_H$ is supported on the fiber $\pi_{r*}^{-1}(H).$
The measure $\lambda \times \ell$ on $\subg \times
  G$ can then be expressed as
  \begin {align*}
    \lambda \times \ell&=\int_{\subg} \delta_H \times \left (\int_{H \backslash G} \eta_{Hg} \, \dd \nu_H(Hg) \right)\dd \lambda(H)\\ &= \int_{\cosg} \delta_H \times \eta_{Hg} \, \dd \nu(Hg),
  \end{align*}
  where $\delta_H$ is the Dirac measure at $H\in \subg$. This is a
  disintegration of $\lambda \times \ell$ with respect to $\sigma_r :
  \subg \times G \longrightarrow \cosg$, $\sigma_r(H,g)=Hg$.

  The left $G$-action on $\subg \times G$ is $(H,g) \overset
  {k}{\mapsto} (k H k^{-1},kg)$. Since $\lambda $ is conjugation
  invariant and $\ell$ is left $G$-invariant, $\lambda\times\ell$ is
  left $G$-invariant. But by the equivariance in Fact \ref{etas}, we have that if $k\in G$,
$$k_*(\delta_H \times \eta_{Hg})=\delta_{k H k^{-1}} \times \eta{(k H
  k^{-1})(k g)}$$ so the left $G$-action on $\subg \times G$ permutes
the fiber measures $\delta_H \times \eta_{Hg}$. Therefore, as
$\lambda\times\ell$ is invariant, and the collection of fiber measures
of its disintegration is equivariant, we have that the factor measure
$\nu$ is invariant (c.f.\ Proposition \ref{uniquefactors}).

  It remains to be shown that $\pi_{l*}(\nu)\equiv
  \lambda$. Let
  $$\psi : \subg \times G \to \subg \times G$$ be given by $\psi(H,g) =
  (H^{g^{-1}},g)$. Then
\begin{align*}
    \pi_l \circ \sigma_r = \pi_r \circ \sigma_r \circ \psi,
  \end{align*} since $\pi_l \circ \sigma_r(H,g)=H^{g^{-1}}$ and $\pi_r \circ \sigma_r$ is the
  projection on the first coordinate.
  But $\psi$ preserves $\lambda \times \ell$, since $\lambda$ is
  conjugation invariant, so
  $$
  \pi_{l*}(\nu) \equiv (\pi_l \circ \sigma_r)_*(\lambda \times \ell) = (\pi_r \circ
  \sigma_r)_*(\lambda \times \ell) \equiv \lambda. \qedhere
  $$
\end{proof}

\begin{remark}\label{longremark}
  The measures $\nu^H$ in Theorem~\ref{thm:unimodular} are only determined up to a positive multiple, so the integrated measure 
  \begin{align*}
\nu=\int_{\subg} \nu^H \, \dd \lambda(H)  \end{align*}
is unique up to scaling by a positive Borel function $f : \cosg \longrightarrow \R$ that is constant on each fiber $\pi_l^{-1}(H)=G/H$. 

Sometimes, but not always, it is possible to choose the measure $\nu$ in Proposition~\ref{prop:equivalent} 
to be invariant under both the left and right actions of $G$ on $\cosg$. If $\lambda$ is ergodic, such a $\nu$ is unique up to a global scalar. This is discussed in Section~\ref{sec:bi-invariant}.
\end{remark}

\section{Invariant measures on orbit equivalence relations} 
As in the introduction, suppose that $G \curvearrowright (X,\zeta)$ is
a measure preserving action on a standard Borel probability space and
let
$$E = \{(x,g x)\,:\,x \in X, g \in G\} \subset X \times X$$ 
be the associated orbit equivalence relation, which we consider with
the $G$-action $g(x,y) = (x,g y)$ and the left projection $p_l \colon E
\to X$, $p_l(x,y) = x$.  The standard fact that $E$ is a Borel subset
of $X \times X$ follows (for example) from the Compact Model Theorem
of Varadarajan~\cite{varadarajan1963groups}*{Theorem 3.2}; see for
example~\cite{zimmer1984ergodic}*{Corollary 2.1.20}.

We aim to prove:

\begin{named}{Corollary \ref{cor:actions}}

    Let $G \curvearrowright (X,\zeta)$ be a measure
  preserving action on a $\sigma $-finite Borel measure space, with
  orbit equivalence relation $E$.

  Then there exists a Borel family of $G$-invariant, $\sigma$-finite
  measures $\nu_x$ on $E$, with $\nu_x$ supported on the fiber
  $p_l^{-1}(x)$, and such that
  \begin{align*}
    \nu = \int_X\nu_x\,\dd\zeta(x)
  \end{align*}
  is a $G$-invariant, $\sigma$-finite measure on
  $E$.
\end{named}

To begin the proof, note that $E$ decomposes as a union of the fibers
$p_l^{-1}(x) = \{x\} \times G x$.  The fiber $\{x\} \times G x$ is
invariant under the $G$-action on $E $, and is isomorphic as a
$G$-space to the quotient $G/G_x$, where $$G_x = \{g \in G\,:\, g x=
x\}$$ is the stabilizer of $x$.  Let $\stab \colon X \to \subg $ be
the map that assigns to each $x \in X$ its stabilizer.  As stabilizers
of $G$-actions on separable Borel spaces are
closed~\cite{varadarajan1963groups}, the image of $\stab$ is in
$\subg$. Since $\stab$ is equivariant, $\stab_*\zeta$ is an invariant subgroup measure.
Theorem~\ref{stronger} then gives a Borel map $$\subg \longrightarrow
\mathcal M (\cosg), \ \ H \mapsto \nu^H$$ that associates to $\stab_*
\zeta$-a.e.\ $H \in \subg$ a nonzero $G$-invariant measure on $G/H
\subset \cosg$.  So, for $\zeta$-a.e.\ $x\in X$ we can define $\nu^x$
to be the measure on $\{x\} \times G x \cong G/G_x$ corresponding to
$\nu^{G_x}$.  The map
$$X \longrightarrow \mathcal M(E), \ \ \ x \mapsto \nu^x$$
is then Borel, so we can integrate the measures $\nu^x $ against the
measure $\zeta$ on $X$ to give a $\sigma$-finite Borel measure $\tau$
on $E $:
\begin{align*}
  \tau = \int_X\nu^x \, \dd\zeta(x).
\end{align*}
This $\tau $ is invariant under the $G$-action on the second
coordinate of $E$, since that action preserves the $p_l$-fibers $\{x\}
\times G x $ and the associated measures $\nu^x$.  Therefore,
Corollary~\ref{cor:actions} follows.

\section{The mass transport principle}
\label{sec:mtp}

We start this section in~\ref{sec:labeled} by an additional discussion
of some well known aspects of mass transport principles and their
relation to invariant random subgroups (cf.\
~\cites{Aldousprocesses,abert2012kesten}). We then prove our theorem
in~\ref{sec:mtp-proof}.

\subsection{The MTP for labeled graphs and discrete groups}
\label{sec:labeled}
  
Fix a finite set $S$.  An \emph {$S$-labeled graph} is a countable directed
graph with edges labeled by elements of $S $, such that the edges
coming out from any given vertex $v$ have labels in 1-1 correspondence
with elements of $S $, and the same is true for the labels of edges
coming into $v$. 
Let $\CG^S_{\bullet}$ and
$\CG^S_{\bullet\bullet}$ be the spaces of rooted and doubly rooted $S$-labeled
graphs, up to (label preserving) isomorphism. The topologies on $\CG^S_{\bullet}$ and $\CG^S_{\bullet\bullet}$ are defined so that two (doubly) rooted, $S$-labeled graphs are close when large finite balls around their roots are isomorphic.

A \emph
{unimodular random $S$-labeled graph} is a random $S$-labeled graph
whose law is a probability measure $\lambda $ on $\CG^S_{\bullet}$
such that for every nonnegative Borel function $f :\CG_{\bullet\bullet}^S
\longrightarrow \R$ we have the \emph {mass transport principle}
\begin{align}
  \label{eq:Sgraph-mtp}
  \int_{\CG^S_\bullet}\ \sum_{w\in G}f(G,v,w)\, \dd\lambda(G,v) =
  \int_{\CG^S_\bullet}\ \sum_{w\in G}f(G,w,v)\, \dd\lambda(G,v).
\end{align}

In the introduction, we mentioned that unimodularity is often
equivalent to invariance under a group action, when an action exists.
Here, we can use the $S$-labels to construct an action of the free
group $F(S)$ generated by $S $ on $\CG ^ S_{\bullet} $, where the
action of $s$ moves the root along the adjacent \emph{inward} edge labeled
`s'.  We then have

\begin {fact}\label{thefact}
  A random $S$-labeled graph is unimodular if and only if it is
  invariant under the action of $F(S) $.
\end {fact}
\begin {proof}
  Assume first that $\lambda $ is the law of a unimodular random
  $S$-labeled graph.  If $s\in S $ and $E \subset \CG ^ S_\bullet $ is
  Borel, define $f: \CG_{\bullet\bullet}^S \longrightarrow \R$ by $$f
  (G,v,w) = \begin {cases} 1_E (G,v) & (G,v) = s(G,w) \\ 0 & \text
    {otherwise} \end {cases} $$ Then the left side of the MTP is
  $\lambda (E) $, while the right-hand side is $\lambda (s(E)) $.
  Therefore, $\lambda $ is $s$-invariant.

  For the other direction, suppose that $\lambda $ is invariant under the
  action of each $s\in S $.  By a standard reduction, it suffices to
  prove the MTP for functions $f:\CG_{\bullet\bullet}^S
  \longrightarrow \R$ supported on graphs $(G, v, w) $ where there is
  a directed edge from $w $ to $v$ (compare with Proposition 2.2 in
  \cite {Aldousprocesses}).  If $f $ is such a function, the left side
  of the MTP becomes
  \begin {align} \int_{\CG^S_\bullet}\ \sum_{s\in S \cup S^ {-
        1}}\frac {f(G,v,s v)} {m (G, x, s)} \, \dd\lambda(G,v).\label
    {mtpfirst}
  \end {align}
  Here, $S ^ {-1 } $ is the set of formal inverses of elements of $S
  $, where the action of $s^{-1}$ moves a vertex of a graph along the
  adjacent \emph {outward} edge labeled `s'.  The multiplicity function
  $m(G,x,s)$ records the number of elements $t \in S \cup S ^ {- 1} $
  where $s x=t x$.  Then
  \begin{align*}
    \eqref{mtpfirst} &= \sum_{s\in S \cup S^ {- 1}} \int_{\CG^S_\bullet}\ \frac {f(G,v,s v)} {m (G, v, s)} \, \dd\lambda(G,v) \\
    & = \sum_{s\in S \cup S^ {- 1}} \int_{\CG^S_\bullet}\ \frac {f(G,s^{-1}v,v)} {m (G, s^{- 1}v, s)} \, \dd\lambda(G,v) \\
    & = \int_{\CG^S_\bullet} \sum_{s\in S \cup S^ {- 1}} \ \frac {f(G,s v,v)} {m (G, s v, s^{-1})} \, \dd\lambda(G,v) \\
    & = \int_{\CG^S_\bullet} \sum_{w\in G} \ {f(G,w,v)} \,
    \dd\lambda(G,v) ,
  \end {align*}
  which proves the mass transport principle.
\end {proof}

In~\cite{abert2012kesten}, Ab{\'e}rt, Glasner and Vir{\'a}g show how
to produce unimodular random $S$-labeled graphs from invariant
subgroup measures.  Suppose that $G $ is a group generated by the
finite set $S$.  The \emph {Schreier graph} of a subgroup $H \leq G$
is the graph $\mathrm{Sch}(H\backslash G,S) $ whose vertices are right cosets of
$H$ and where each $s \in S$ contributes a directed edge labeled `s'
from every coset $Hg$ to $H g s$.  We consider $\mathrm{Sch}(H\backslash G,S)$ as
an $S$-labeled graph rooted at the identity coset $H$, in which case the action of $F(S)$ described above is $Hg \overset {s}{\longmapsto} H g s^{-1}$. This defines
an injection
$$\Phi : \subg \longrightarrow \CG^ S_\bullet, \ \ H \longmapsto (\mathrm {Sch} (H\backslash G, S), H). $$
Under $\Phi $, the conjugation action of an element $s \in S$ on
$\subg $ corresponds to the natural action on $\CG_\bullet ^ S $,
since 
$$(\mathrm{Sch}(H \backslash G,S),H s^{-1}) \cong (\mathrm{Sch}(s H s^{-1}\backslash G,S), s H s^{-1})$$
 as $S$-labeled rooted
graphs. Therefore, invariant subgroup measures of $G $ produce $F (S)
$-invariant measures on $\CG ^ S_\bullet $.

We can also use the map $\Phi $ to reinterpret the MTP for $S
$-labeled graphs~\eqref{eq:Sgraph-mtp} group theoretically.  Let
$\cosg $ be the space of cosets of subgroups of $G $, as in Section
\ref{proof}. There is then an injection
\begin {align} 
 \Phi':  \cosg  \longrightarrow \CG^S_{\bullet\bullet} , \ \ Hg \longmapsto (\mathrm{Sch}(H \backslash  G,S), H, Hg).
\end {align}
To interpret the right hand side of \eqref {eq:Sgraph-mtp}, note that
the doubly rooted graph $(\mathrm{Sch}(H\backslash G,S),  Hg,H)$ that arises
from switching the order of the roots in $\Phi (H) $ is isomorphic as
an $S$-labeled graph to
$$(\mathrm{Sch}(g^ {- 1} H g \backslash G,S), g ^ {- 1}H g , g ^ {- 1} H g g^ {- 1}) = \Phi'( (g^{-1}Hg)g ^ {- 1})=\Phi'( g^{-1}H).$$  

Therefore, Fact \ref{thefact} translates under $\Phi $ and $\Phi' $
to the following characterization of invariant subgroup measures of a
discrete group $G $.

\begin {named}{The Discrete MTP} Suppose that $G $ is a finitely generated
  group and $\lambda$ is a Borel measure on $\subg $.  Then $\lambda $
  is conjugation invariant if and only if for every nonnegative Borel function $ f:\cosg
  \longrightarrow \R$,
  \begin{align*}
    \int_{\subg}\sum_{H\backslash G}f(Hg)\, \dd\lambda(H) =
    \int_{\subg}\sum_{H\backslash G}f(g^{-1}H)\, \dd\lambda(H).
  \end{align*}
\end {named}

Re-indexing, a slightly more aesthetic statement of the discrete MTP is obtained by replacing the sum on the right with $\sum_{G/H} f(g H)$.

\subsection{Proof of the Mass Transport Theorem}
\label{sec:mtp-proof}  
We now extend the Discrete MTP to general $G$ and discrete or compact
$\lambda$. As the following is intimately related to the existence of an invariant measure on $\cosg$, we will frequently reference the setup of Sections \ref{sec:disintegration} and \ref{proof}. 

Suppose that $G $ is a locally compact, second countable topological
group, and fix a left invariant Haar measure $\ell$ on $G $. If $H$ is a compact subgroup of $G $, let $\nu_H$ be the push forward of $\ell$ to $H\backslash G$.
If $H $ is a discrete subgroup of $G$, let $\nu_H$ be the measure on $H\backslash G$
obtained by \emph{locally} pushing forward $\ell$ under the covering map $G
\longrightarrow H\backslash G$. That is, if $U\subset H\backslash G$ is an evenly covered open set with preimage $V_1 \sqcup V_2 \sqcup \cdots $, then $\nu_H | _U$ is the push forward of $\ell |_{V_i}$ for every $i$.

These $\nu_H$ can be understood in terms of Proposition \ref{prop:constructing-nu}.  For discrete cosets, set $\eta_{Hg} $ to be the counting measure. When $H$ is compact, let $\eta_{Hg} $ be the unique left $H$-invariant probability measure on $Hg$. Then in both cases, $\nu_H$ is characterized by the equation
\begin {align}
  \label{pushforward}
  \ell = \int_{H \backslash G} \eta_{H g} \, \dd\nu _ H.
\end {align}

Note that by Proposition \ref{prop:constructing-nu}, the map $H \longmapsto \nu_H$ from $\subg$ to the space of measures on $\cosg$ is Borel.

\begin{theorem}[Mass Transport Principle]
  \label{thm:mtp}
  Let $\lambda$ be a $\sigma$-finite Borel measure on $\subg$ such
  that $\lambda$-a.e.\ $H\in\subg $ is discrete or compact.
 
Then $\lambda $ is
  conjugation invariant (i.e., an invariant subgroup measure) if and
  only if $\mu_G |_H = 1 $ for $\lambda$-a.e.\ $H\in\subg $, and  for every nonnegative Borel function $f \colon \cosg \rightarrow \R,$
  \begin{align}
    \label{eq:mtp}
    & \ \ \ \int_{\subg} \int_{H\backslash G} f(Hg)\, \mu_G(Hg)\, \dd\nu_H(Hg) \,
    \dd\lambda(H) \nonumber \\
    &=\int_{\subg} \int_{H\backslash G}f(g^{-1}H) \,
    \dd\nu_H(Hg)\, \dd\lambda(H).
  \end{align}
\end{theorem}

Here, $\mu_G$ is the modular function of $G $, defined by the equation $$ \ell(S)=\mu_G(g) \ell(S g), \ \ \text {for every Borel } S\subset G. $$
The assumption $\mu_G |_H = 1 $ implies that $\mu_G$ is constant over each coset $Hg$, and we write $\mu_G(Hg)$ for the common value on that coset. Note that the modular function of a discrete or compact group is trivial, so Theorem \ref{thm:unimodular} implies that if $\lambda $ is a discrete or compact invariant subgroup measure, then $\mu_G |_H=\mu_H = 1 $ automatically for $\lambda$-a.e.\ $H$. So, in the course of the proof below, we will always assume $\mu_G |_H=\mu_H = 1 $.

The MTP can be stated in a way that is more similar to our work earlier in the paper. Define an involution
$$\rho \colon
  \cosg \longrightarrow \cosg, \ \ \rho(Hg) = g^{-1}H,$$
and let $\nu$ be the measure on $\cosg$ defined by the integral
$$\nu := \int_{\subg} \nu_H \, \dd \lambda(H).$$
Note that the map $Hg \longmapsto \eta_{Hg}$ defined above Equation
\eqref{pushforward} is left $G$-equivariant, so just as in the proof
of Proposition \ref{prop:equivalent} the measure $\nu$ is left
$G$-invariant. 
Now, changing variables on the right hand side, Equation \eqref{eq:mtp} can be rewritten as:
\begin{align}
  \label{eq:mtp2}
  \int_{\cosg } f(Hg) \mu(Hg)\, \dd\nu(Hg) = \int_{\cosg} f(Hg) \, \dd\rho_*\nu(Hg).
\end{align}

So, the MTP  reduces to the following claim.

\begin {claim}\label{mtpclaim}
 $\lambda $ is conjugation invariant if and only if $ \frac 1 {\mu_G}\cdot \nu = \rho_*\nu.$
\end {claim}

\begin {proof}[Proof of Claim \ref{mtpclaim}]
Recall the following maps defined in Section \ref{sec:disintegration}:
  \begin{equation*}
    \begin{array}{r c l  c l}
      \subg \times G& \overset{\sigma_r}{\longrightarrow} &\cosg & \overset {\pi_r}{\longrightarrow} & \subg,\\
     (H,g)&\longmapsto     &Hg &\longmapsto     & H.
   \end{array}
 \end{equation*}
As in the proof of Proposition \ref{prop:equivalent}, we have a $\sigma_r$-disintegration
\begin {equation}\label{mtpdis}
\lambda \times \ell= \int_{\cosg} \delta_H \times \eta_{Hg} \ \dd \nu(Hg),
\end {equation}
where $\ell$ is our chosen left Haar measure on $G$.  Since $\rho(Hg) = g ^ {- 1}H  = H^{g^{-1}} g^{-1}$, there is a
  commutative diagram
  $$\xymatrix {   \subg \times G \ar[r]^{\varphi} \ar[d]^{\sigma_r} &  \subg\times G \ar[d]^{\sigma_r} \\ \cosg \ar[r]^{\rho} &\cosg},$$
  where $\varphi $ is defined by $ \varphi( H,g) =
  (H^{g^{-1}},g^{-1})$. Note that if
$$\iota: G \longrightarrow G, \ \ \iota(g)=g^{-1},$$
then $\iota_*(\ell)=\frac 1{\mu_G} \cdot \ell$, so as $\varphi$ inverts the second factor in $\subg \times G$ and conjugates the first, \emph{$\lambda $ is conjugation invariant if and only if} 
\begin {equation}\label{vareq}
\varphi_*(\lambda \times \ell)(H,g)=\frac 1{\mu_G(g)}\, \lambda \times \ell(H,g).
\end {equation}

For each $Hg\in \cosg$, the measure $\eta_{Hg}$ is either the counting
measure or the unique $H$-invariant probability measure on $Hg$. In
addition to being left $H$-invariant, in both cases $\eta_{Hg}$ is in
fact invariant under the right action of $H^{g^{-1}}$ on $Hg$. This is
immediate when it is a counting measure, and when $H$ is compact,
$\eta_{Hg}$ is the push forward under left multiplication by $g$ of
the \emph{bi-invariant} Haar probability measure on the compact group
$H^{g^{-1}}$.  So, the pushforward of $\eta_{Hg}$ under inversion $g
\longmapsto g^{-1}$ is a \emph{left} $H^{g^{-1}}$-invariant measure on
$H^{g^{-1}}g^{-1}$, and must be $\eta_{H^{g^{-1}}g^{-1}}$. Multiplying
by $\delta_H$, we then have
  \begin{align*}
    \varphi_*(\delta_H \times \eta_{Hg} ) = \delta_{H ^ {g^{-1}}} \times\eta_{H^{g^{-1}} g^{-1}}. 
  \end{align*}
So, the fiber measures of the disintegration in \eqref{mtpdis} are permuted by $\varphi$. From the commutative diagram for $\varphi$ and $\rho$, it follows that $\rho$ scales $\nu$ by $1/\mu_G$ if and only if $\varphi$ scales $\lambda\times\ell$ by $1/\mu_G$, which we saw above was equivalent to conjugation invariance of $\lambda $.
\end{proof}

\subsection{Bi-invariant measures and co-unimodular IRSs}
\label{sec:bi-invariant}
As a consequence of Claim \ref{mtpclaim}, we record the following.

\begin {corollary}\label{bi-invariant}
Suppose that $\lambda $ is an invariant subgroup measure in a unimodular group $G $ such that $\lambda $-a.e.\ $H\in\subg$ is discrete or cocompact. Then there is a measure $\nu$ on $\cosg$ that is invariant under both the left and right actions of $G $, and for which $\pi_{l*}(\nu)\equiv\pi_{r*}(\nu)\equiv\lambda $. Moreover, if $\lambda $ is ergodic then $\nu$ is unique up to scale.
\end{corollary}
\begin {proof}
The construction above gives a measure $\nu$ on $\cosg$ that is left $G $-invariant and also $\rho$-invariant. But the map $\rho(Hg) = g^{-1}H$ conjugates the left $G $-action on $\cosg$ to the right $G$-action, so $\nu$ is also right $G $-invariant. The fact that $\pi_{r*}(\nu)\equiv\lambda $ is just the definition of $\nu$, and the fact that $\pi_{l*}(\nu)\equiv\lambda $ is the argument at the end of Proposition \ref{prop:equivalent}.

Since $\nu$ is bi-invariant, it is preserved under conjugation by $k\in G$, i.e.\ by the map $$\cosg \longrightarrow \cosg, \ \ g H \overset {k}{\longrightarrow } k g H k=k g k^{-1}(k H k^{-1}).$$ This map sends $G/H$ to $G/(k H k^{-1})$, so permutes the fibers of $\pi_l$.  As it also preserves $\nu$ and its $\pi_l$-factor measure $\lambda $, the conjugation map must permute the fiber measures $\nu^H$. Therefore, if $\nu$ and $\nu'$ are both bi-invariant and $\lambda $ is ergodic, the function
$$H \longrightarrow \frac {\dd \nu^H}{\dd (\nu')^H}$$
is conjugation invariant, so is constant on a $\lambda $-full measure set. This implies that $\nu$ and $\nu'$ agree up to a scalar multiple, as desired.
\end {proof}

Bi-invariant $\nu$ do not always exist for general invariant subgroup measures $\lambda $, even in unimodular groups $G$.  For instance, if $\lambda $ is an atomic measure on a normal subgroup $N\lhd G$, then $\nu=\nu_{N}$ is just the left Haar measure on $G/N\subset\cosg$, which is right invariant exactly when $G/N$ is unimodular. Note that unimodular $G$ may have non-unimodular quotients $G/N$: an example is
$$\mathrm{Sol}=\R^2 \rtimes \R, \ \ t\in \R \longmapsto \begin {pmatrix} e^t & 0 \\ 0 & e^{-t} \end{pmatrix} \circlearrowright \R^2,$$
where $N$ is the $x$-axis in $\R^2 \subset \mathrm{Sol}$. Although $\mathrm{Sol}$ is unimodular, the quotient $\mathrm{Sol}/N$ is the group of affine transformations $$\mathrm{Aff}(\R)=\{x\mapsto ax+b \ | \ a,b\in \R\},$$ which is not unimodular. 

As suggested in the introduction, it is natural to call an IRS (or
more generally, an invariant subgroup measure) $\lambda$ co-unimodular
if there exists a bi-invariant $\nu$ such that
$\pi_{l*}(\nu)\equiv\pi_{r*}(\nu)\equiv\lambda $. Using this
terminology, Corollary~\ref{bi-invariant} states that inside unimodular $G$, discrete IRSs
are co-unimodular, as are compact ones. In general, co-unimodularity is equivalent to obeying a mass transport principle that is not twisted by the modular function.

\begin{claim}
  \label{clm:co-unimodular}
  Let $\lambda $ be an invariant subgroup measure in a locally compact
  second countable group $G$. Then the following are equivalent.
  \begin{enumerate}
  \item There exists a $G$-left-invariant measure $\nu$ on $\cosg$
    such that $\pi_{l*}(\nu)\equiv\pi_{r*}(\nu)\equiv\lambda $ and
    such that $\rho_*\nu = \nu$ (i.e.\ an untwisted MTP holds for $\nu$;
    see~\eqref{eq:mtp2} without the $\mu_G$ factor).
  \item There exists a $G$-bi-invariant measure $\nu$ on $\cosg$ such
    that $\pi_{l*}(\nu)\equiv\pi_{r*}(\nu)\equiv\lambda $ (that is,
    $\lambda$ is co-unimodular).
  \end{enumerate}
\end{claim}

\begin{proof}
  Fix a left-invariant measure $\nu$ such that
  $\pi_{l*}(\nu)\equiv\pi_{r*}(\nu)\equiv\lambda $. If $\rho_*\nu=\nu$
  then $\nu$ is bi-invariant, following the argument of
  Corollary~\ref{bi-invariant}. Conversely, if $\nu$ is bi-invariant,
  then $\nu + \rho_*\nu$ has the same properties of $\nu$, but is also
  $\rho$-invariant.
\end{proof}

\appendix

\section{Some notes on Haar measures}

Suppose that $G $ is a locally compact, second countable group and fix
throughout this section a continuous, nonnegative function $f : G \to
\R$ with compact support such that $f(1)=1$.

\begin {definition}
  If $H $ is a closed subgroup of $G $, we define $m_f(H)$ to be the
  unique left Haar measure on $H$ such that $\int f \, \dd m_f (H) = 1$.
\end {definition}

The measure $m_f(H)$ will exist and be unique as long as we have that
the integral of $f $ against any left Haar measure $\mu $ is nonzero
and finite. The former follows since $\mu$ is positive on open subsets
of $H$ and $f$ is positive on a neighborhood of the identity.  The
latter holds since $\mu$ is Radon and $\supp{f|_H}$ is compact.

\begin{claim}\label{haarmeasures}
  The map $m_f : \subg \to \cM(G)$ is continuous, where $\subg$ has
  the Chabauty topology and $\cM(G)$ the topology of weak convergence.
\end{claim}
\begin{proof}
  Let $H_i \to H$ in $\subg$ and set $m_i = m_f(H_i)$.  Fix any
  subsequence $(m_{i_j})$ of $m_i$.  By~\cite[Proposition
  2]{Bourbakiintegration}, the space of Haar measures on closed
  subgroups of $G$ with the normalization $\int f \, \dd m_f (H) = 1$ is
  weak* compact.  Therefore, $(m_{i_j})$ has a further subsequence
  that converges to some Haar measure $m$ for a closed subgroup of
  $G$.

  We claim that $m $ is $H$-invariant. Let $h\in H$ and take a
  sequence of elements $h_{i_j} \in H_{i_j} $ with $h_{i_j} \to h $.
  Then for any continuous function $g : G \to \R$ with compact
  support, we have
  \begin {align*}
    \int g(h x) \, \dd m(x) &= \lim_j \int g(h_{i_j}x) \, \dd m_{i_j}(x) \\
    &= \lim_j \int g(x) \, \dd m_{i_j}(x) \\
    &= \int g(x) \, \dd m (x).
  \end {align*}
  The justification for the double limit is that $g(h_{i_j} \, \cdot )
  \to g( h \, \cdot) $ converge uniformly and all have support within
  some compact $K $, on which we have $m_{i_j}(K)$ bounded
  independently of $j $.  It follows that $m $ is $H$-invariant.

  We now claim that $m$ is supported within $H$.  For given $g \in G
  \setminus H$, let $U$ be an open set that is disjoint from some
  neighborhood of $H$.  By the definition of the Chabauty topology, $U
  \cap H_{i_j} = \emptyset$ for large $j $.  Therefore, $m_{i_j}(U) =
  0$ for large $j$, implying that $m (U) = 0 $.

This shows that $m $ is a left Haar measure on $H $.  Finally, as $$\int f \,
  \dd m=\lim_{j} \int f \, \dd m_{i_j} = \lim_j 1 = 1,$$ it must be that
  $m=m_f(H)$, so the claim follows.
\end{proof}

\begin {lemma}\label{closed}
  Suppose that $G $ is a locally compact topological group. Then the
  subset $\mathcal U \subset \subg$ consisting of subgroups $H $ for
  which $G/H$ has an invariant measure is closed.
\end {lemma}
\begin {proof}
  The space $G/H$ admits a $G$-invariant measure if and only if
  the modular function of $G $ restricts to the modular function of $H
  $: that is, $\mu_G (h) =\mu_H (h) $ for all $h\in H$ (see,
  e.g.,~\cite{nachbin1965}).  So, suppose that we have a sequence of
  elements $H_i \in \mathcal U$ and that $H_i \to H$.  If $h\in H,$ we
  want to show that $\mu_G (h) = \mu_H (h)$.  Let $h_i\in H_i $ with
  $h_i \to h$.  Then
  \begin {align*}
    h_* m_f(H) &= \lim_i \, (h_i)_* m_f(H_i) \\
    &= \lim_i \mu_{H_i} (h_i) m_f(H_i) \\
    &= \lim_i \mu_G (h_i) m_f(H_i) \\
    &= \mu_G(h) m_f(H),
  \end {align*}
  so $\mu_G (h) = \mu_H (h)$ and the lemma follows.
\end {proof}

\label{haar}

\section{Uniqueness of factor measures}

We give here a proof of the following standard uniqueness statement
for factor measures in a disintegration with prescribed non-zero fiber
measures.

\begin {proposition}\label{uniquefactors}
  Let $X $ and $Y$ be Borel spaces, $p: X \longrightarrow Y $ be a
  Borel map and $$Y \longrightarrow \mathcal M (X), \ \ y \longmapsto
  \eta_y$$ be a Borel parametrization of a family of $\sigma $-finite
  Borel measures on $X $ such that each $\eta_y $ is nonzero and
  supported on $p ^ {- 1} (y) $.  If $\mu $ and $\mu' $ are $\sigma
  $-finite Borel measures on $Y $, let
$$\lambda =\int_Y \eta_y \, \dd\mu (y), \ \ \lambda' = \int_Y \eta_y\, \dd\mu' (y) . $$  Then $\lambda =\lambda' $ if and only if $ \mu =\mu' .$
\end {proposition}
\begin {proof}
  The backwards implication is immediate.  So, assume $\lambda
  =\lambda' $.  We claim that there is a Borel function $f: X \to \R $
  such that
  $$\int_X f  \, \dd\eta_y = 1 $$
  for both $\mu $-a.e.\ and $\mu' $-a.e.\ $y\in Y $.  This will finish
  the proof, since if $U\subset Y $ is Borel then
  $$\mu (U) = \int_X f \cdot 1_U ({p (x)} ) \, \dd\lambda (x) =  \int_X f \cdot 1_U ({p (x)} ) \, \dd\lambda' (x) = \mu' (U) .$$

  The measure $\lambda $ is $\sigma $-finite, so let $U_1 \subset U_2
  \subset \cdots $ be an increasing sequence of Borel subsets of $X $
  with $\lambda (U_i) < \infty$ and $\cup_i U_i = X $.  For each $y\in
  Y $, let $n (y) $ be the minimum $i $ such that $\eta_y (U_i) >0 $.
  Define
  $$f: X \to \R, \ \ f (x)= \sum_{y\in Y} \frac {1_{U_{n (y)} \cap p ^ {- 1} (y)}(x) }  { \eta_y (U_{n (y)})}. $$
  Then $f$ is Borel, and the claim will follow if we show that
  $\eta_y(U_{n (y)}) <\infty $ for both $\mu $-a.e.\ and $\mu' $-a.e.\
  $y\in Y $.  Assume by contradiction that there exists a $V \subseteq Y $ with
  $\eta_y(U_{n(y)}) =\infty $ for all $y\in V $, and, say, $\mu (V) >0
  $.  It follows that there exists a $W \subseteq V$, with $\mu(W)>0$,
  and an $N$ such that $\eta_y(U_N) =\infty $ for all $y\in W $. But
  then
  $$\lambda (U_N) = \int_Y \eta_y (U_N) \, \dd\mu \geq \int_W \eta_y(U_N) \, \dd\mu = \infty, $$
  which contradicts the initial choice of the sets $U_i $.
\end {proof}

\bibliography{manifolds}
\end{document}